\documentclass{amsart}
\usepackage{amsmath}
\usepackage{amsxtra}
\usepackage{amscd}
\usepackage{amsthm}
\usepackage{amsfonts}
\usepackage{amssymb}
\usepackage{mathrsfs}
\usepackage{eucal} 
\usepackage[hidelinks, pagebackref]{hyperref}
\hypersetup{
    colorlinks,
    citecolor=black,
    filecolor=black,
    linkcolor=black,
    urlcolor=black,
}
\usepackage{cite}
\usepackage{microtype}
\usepackage{url}
\input xy
\xyoption{all}

\newtheorem{thm}{Theorem}
\newtheorem{cor}[thm]{Corollary}
\newtheorem{lem}[thm]{Lemma}

\newtheorem{conj}[thm]{Conjecture}

\theoremstyle{remark}
\newtheorem{rem}[thm]{Remark}

\theoremstyle{definition}

\newcommand\nc{\newcommand}
\nc\on{\operatorname}

\newcommand\fp{{\mathfrak p}}
\newcommand\fq{{\mathfrak q}}
\newcommand\fm{{\mathfrak m}}

\newcommand\mbf{\mathbf}

\nc\Hom{\on{Hom}}
\nc\Sections{\on{Sections}}
\nc\Sym{\on{Sym}}
\nc\Spec{\on{Spec}}
\nc\Specm{\on{Specm}}
\nc\ul{\underline}
\nc{\dfp}{\overset{\cdot}{\fp}}
\nc{\dfq}{\overset{\cdot}{\fq}}
\nc{\dfm}{\overset{\cdot}{\fm}}

\begin{document}

\title{Manifolds Containing an Ample $\mathbb{P}^1$-Bundle}
\author{Daniel Litt}

\begin{abstract}
Sommese has conjectured a classification of smooth projective varieties $X$  containing, as an ample divisor, a $\mathbb{P}^d$-bundle $Y$ over a smooth variety $Z$.  This conjecture is known if $d>1$, if $\dim(X)\leq 4$, or if $Z$ admits a finite morphism to an Abelian variety.  We confirm the conjecture if the Picard rank $\rho(Z)=1$, or if $Z$ is not uniruled.  In general we reduce the conjecture to a conjectural characterization of projective space: namely that if $W$ is a smooth projective variety, $\mathscr{E}$ is an ample vector bundle on $W$, and $\text{Hom}(\mathscr{E}, T_W)\not=0$, then $W\simeq \mathbb{P}^n$.
\end{abstract}
\maketitle

\tableofcontents
\section{Introduction}
Beltrametti and Sommese give a conjectural classification of smooth projective varieties $X$ containing a $\mathbb{P}^d$-bundle as an ample divisor \cite[Conjecture 5.5.1]{adjunction-theory}.  The main goal of this paper is to prove this conjecture in the case where $X$ has minimal Picard rank, $$\rho(X)=2.$$ 
 
Throughout the paper we work over $\mathbb{C}$; the phrase ``$\mathbb{P}^d$-bundle," will be taken to mean a $\mathbb{P}^d$-bundle locally trivial in the analytic topology.

The conjecture is:
\begin{conj}\label{sommeseconj}
Let $X$ be a smooth projective variety and $Y\subset X$ a smooth ample divisor.  Suppose that $p: Y\to Z$ is a morphism exhibiting $Y$ as a $\mathbb{P}^d$-bundle over a $b$-dimensional manifold $Z$.  Then one of the following holds:
\begin{enumerate}
\item $X\simeq\mathbb{P}^3$, $Y\simeq \mathbb{P}^1\times \mathbb{P}^1$ is a smooth quadric, and $p$ is one of the projections to $\mathbb{P}^1$.
\item $X\simeq {Q}^3\subset \mathbb{P}^4$ is a smooth quadric threefold, $Y\simeq \mathbb{P}^1\times \mathbb{P}^1$ is a hyperplane section, and $p$ is a projection to one of the factors.
\item $Y\simeq \mathbb{P}^1\times \mathbb{P}^b$, $Z\simeq \mathbb{P}^b$, $p: Y\to Z$ is the projection to the second factor, and $X$ is the projectivization of an ample vector bundle $\mathscr{E}$ on $\mathbb{P}^1$.
\item $X\simeq \mathbb{P}(\mathscr{E})$ for an ample vector bundle $\mathscr{E}$ on $Z$, and $\mathscr{O}_X(Y)\simeq \mathscr{O}_{\mathbb{P}(\mathscr{E})}(1)$ (i.e. $Y$ is a fiberwise hyperplane).
\end{enumerate}
\end{conj}
Sommese has proven the case where $d\geq 2$ (see e.g. \cite[Theorem 5.5.2]{adjunction-theory}).  The cases where $d=1$ and $b=1, 2$ are due to the work of several authors (see e.g. \cite[Theorem 7.4]{beltrametti-ionescu} and the references therein).  We prove the conjecture in the case where $$\rho(Z)=1,$$ (if $\dim(X)\geq 4$, this is equivalent to $\rho(X)=2$) and in general reduce it to a plausible conjectural improvement of a result of Andreatta-Wisniewski \cite{andreatta-wisniewski}, namely
\begin{conj}\label{awconj}
Let $X$ be a smooth projective variety and $\mathscr{E}$ an ample vector bundle on $X$.  If $$\on{Hom}(\mathscr{E}, T_X)\neq 0,$$ then $X\simeq \mathbb{P}^n$.
\end{conj}
We also prove the conjecture in the case that $Z$ is not uniruled.
\begin{rem}
By \cite{andreatta-wisniewski}, the existence of a map $\mathscr{E}\to T_X$ of constant rank implies $X\simeq \mathbb{P}^n$; likewise, \cite[Corollary 4.3]{kebekus} proves the conjecture if $\rho(X)=1$.  One may also use the methods of \cite[Section 4]{kebekus} to prove the conjecture if there exists a map $\mathscr{E}\to T_X$ generically of maximal rank.
\end{rem}
The idea of our argument is to show (via an analysis of the deformation theory of the map $p: Y\to Z$) that either $p$ extends to a map $\tilde{p}: X\to Z$ (using results of \cite{litt}), or there is an ample vector bundle $\mathscr{E}$ on $Z$ and a map $\mathscr{E}\to T_Z$.  In the former case, we are done by work of Sommese; in the latter case, we're may apply Conjecture 2 to proceed.

{\bf Acknowledgments.}
This note owes a debt to conversations with Tommaso de Fernex, Paltin Ionescu, and Jason Starr.  It was written with support from an NSF Postdoctoral Fellowship.
\section{The Proof}
We first show:
\begin{lem}\label{mainlemma}
As before, let $X$ be a smooth projective variety, $Y\subset X$ a smooth ample divisor, and $p: Y\to Z$ a $\mathbb{P}^1$-bundle.  Let $\widehat{Y}$ be the formal scheme obtained by completing $X$ at $Y$.  If $p: Y\to Z$ does not extend to a morphism $\widehat{p}: \widehat{Y}\to Z$, then there exists an ample vector bundle $\mathscr{E}$ on $Z$ and a non-zero morphism $\mathscr{E}\to T_Z$.
\end{lem}
\begin{proof}
Let $\mathscr{I}_Y$ be the ideal sheaf of $Y$, and let $Y_n$ be the subscheme of $X$ cut out by $\mathscr{I}_Y^n$.  Then the obstruction to extending a map $$p_n: Y_n\to Z$$ to a map $$Y_{n+1}\to Z$$ lies in 
$$\on{Ext}^1(p^*\Omega^1_Z, \mathscr{I}_{Y}^n/\mathscr{I}_Y^{n+1})=H^1(Y, p^*T_Z\otimes \mathscr{I}_Y^n/\mathscr{I}_Y^{n+1}).$$
This last is equal to $$H^1(\mathbf{R}p_*(p^*T_Z\otimes \mathscr{I}^n_Y/\mathscr{I}_Y^{n+1}))$$ which is the same as $$H^1(T_Z\otimes \mathbf{R}p_*\mathscr{I}^n_Y/\mathscr{I}_Y^{n+1})$$ by the projection formula.  As $\mathscr{I}^n_Y/\mathscr{I}_Y^{n+1}$ is anti-ample, this last equals $$H^0(T_Z\otimes \mbf{R}^1p_*\mathscr{I}^n_Y/\mathscr{I}_Y^{n+1}).$$  Applying Serre duality, we see that this is the same as $$H^0(T_Z\otimes(p_*(\mathscr{O}(nY)|_Y\otimes \omega_{Y/Z}))^\vee)=\on{Hom}(p_*(\omega_{Y/Z}\otimes \mathscr{O}(nY)|_Y), T_Z).$$

But by \cite[Theorem 1.2]{mourougane}, $$p_*(\omega_{Y/Z}\otimes \mathscr{O}(nY)|_Y)$$ is either zero or ample.  Thus either the problem of extending $p$ to $\widehat{Y}$ is unobstructed, or the obstruction \emph{is} a non-zero map from an ample vector bundle $\mathscr{E}$ on $Z$ to $T_Z$, as desired.
\end{proof}
We will also require:
\begin{lem}\label{lefschetzlemma}
Let $X$ be a smooth projective variety of dimension at least $3$, and $Y\subset X$ an ample divisor.  Let $Z$ be a smooth variety with $\dim(Z)<\dim(Y)$.  Then the restriction map $$\on{Hom}(X, Z)\to \on{Hom}(\widehat{Y}, Z)$$ is a bijection.  Here $\widehat{Y}$ is, as before, the formal scheme obtained by completing $X$ at $Y$.
\end{lem}
\begin{proof}
This is a combination of two results from \cite{litt}.  First, by \cite[Corollary 2.10]{litt}, applied to the projection $X\times Z\to X$, a map $p: \widehat{Y}\to Z$ extends uniquely to some Zariski-open neighborhood $U$ of $Y$.  Second, by \cite[Corollary 3.3]{litt}, this rational map to $Y$ is in fact regular.
\end{proof}
\begin{cor}
Let $X, Y, Z, p$ be as in Conjecture \ref{sommeseconj}.  Suppose that either 
\begin{enumerate}
\item $Z$ is not uniruled, or 
\item $\rho(Z)=1$ (equivalently, $\rho(X)=\rho(Y)=2$).
\end{enumerate}
Then Conjecture \ref{sommeseconj} is true for $X, Y, Z, p$.
\end{cor}
\begin{proof}
Without loss of generality, $p$ has relative dimension $1$ (i.e. it exhibits $Y$ as a $\mathbb{P}^1$-bundle over $Z$) as the case of relative dimension greater than $1$ is already known \cite[Theorem 5.5.2]{adjunction-theory}.  We may also assume $\dim(Z)>2$, as again, if $\dim(Z)\leq 2$, the result is already known \cite[Theorem 7.4]{beltrametti-ionescu}.

By Lemma \ref{mainlemma}, either $p$ extends to a map $\hat p: \widehat{Y}\to Z$ or $T_Z$ contains an ample subsheaf, namely the image of $\mathscr{E}$ from Lemma \ref{mainlemma}.   In the former case, the map $p$ extends to a map $\tilde p: X\to Z$ by Lemma \ref{lefschetzlemma} and we are done by \cite[Theorem 5.5(ii)]{beltrametti-ionescu} (in particular, we are in case (4) of the conjecture).  In the latter case, we consider the situations 
\begin{enumerate}
\item $Z$ not uniruled, or
\item $\rho(Z)=1$  
\end{enumerate}
separately.

\begin{enumerate}
\item Suppose $Z$ is not uniruled.  Then $T_Z$ contains no ample subsheaves by a result of Miyaoka (see e.g. \cite[IV.1.16]{kollar}), so we have a contradiction. 
\item Alternately, suppose $\rho(Z)=1$.  Then as $T_Z$ contains an ample subsheaf, by \cite[Corollary 4.3]{kebekus}, $Z\simeq \mathbb{P}^n$.  By \cite[Theorem 2.1]{fania1987} we conclude the result, namely that we are in case (3) of the conjecture.
\end{enumerate}
\end{proof}
\begin{cor}
Suppose that Conjecture \ref{awconj} is true.  Then Conjecture \ref{sommeseconj} holds as well.
\end{cor}
\begin{proof}
This is the same argument as in the $\rho(Z)=1$ case above, replacing the reference to \cite{kebekus} with Conjecture \ref{awconj}.
\end{proof}
\bibliographystyle{alpha}
\bibliography{sommese-bibtex}

\end{document}